\documentclass[12pt,a4paper]{amsart}
\usepackage{amsmath,amsfonts,amsthm,amsopn,color,amssymb,enumitem}
\usepackage[a4paper]{geometry}
\usepackage{palatino}
\usepackage{graphicx}
\usepackage[colorlinks=true]{hyperref}
\hypersetup{urlcolor=blue, citecolor=red, linkcolor=blue}
\allowdisplaybreaks

\usepackage{cite}
\usepackage{relsize}
\usepackage{esint}
\usepackage{verbatim}
\usepackage{mathrsfs}
\usepackage{xcolor}
\usepackage{xfrac}

\newcommand{\e}{\varepsilon}
\newcommand{\emb}{\hookrightarrow}

\newcommand{\R}{\mathbb{R}}
\newcommand{\N}{\mathbb{N}}

\newcommand{\de}{\partial}

\DeclareMathOperator{\dv}{div}
\DeclareMathOperator{\tr}{tr}

\DeclareMathOperator{\dist}{dist}

\newcommand{\Sph}{{\mathbb{S}}}

\newcommand{\HH}{\mathcal{H}}

\newcommand{\hsp}{\hspace{0.2cm}}

\newcommand{\intI}{\int_{\Sph^4_+}}
\newcommand{\intB}{\int_{\Sph^3}}
\newcommand{\fixed}{\mathcal F_\de}

\renewcommand{\(}{\left(}
\renewcommand{\)}{\right)}

\newcommand{\sff}{\mathbb I}

\newcommand*{\scal}[1]{\left\langle #1 \right\rangle}

\newcommand*{\abs}[1]{\left\vert #1\right\vert}
\newcommand*{\norm}[1]{\left\Vert #1\right\Vert}

\newtheorem{theorem}{Theorem}[section]
\newtheorem*{theorem*}{Theorem}
\newtheorem{lemma}[theorem]{Lemma}

\newtheorem{proposition}[theorem]{Proposition}

\theoremstyle{definition}

\title[Conformal metrics with symmetric $Q$ and $T$ curvatures]{Conformal metrics on the four-dimensional half sphere with symmetric $Q$ and $T$ curvatures}

\author{Sergio Cruz-Blázquez}
\address{Sergio Cruz Blázquez, Departamento de Análisis Matemático, Universidad de Granada, Avenida de Fuente Nueva s/n, 18071 Granada (Spain)}
\email{sergiocruz@ugr.es}
\author{Azahara DeLaTorre}
\address{Azahara DeLaTorre, Dipartimento di Matematica Guido Castelnuovo, Sapienza Università di Roma, Piazzale Aldo Moro 5, 00185 Roma (Italy)}
\email{azahara.delatorrepedraza@uniroma1.it}

\begin{document}
\maketitle
\begin{abstract}
In this paper, we address the problem of prescribing non-constant $Q$ and boundary $T$ curvatures on the upper hemisphere $\Sph^4_+\subset \R^5$, via a conformal change of the background metric. This is equivalent to solve a fourth-order non-linear elliptic boundary value problem with a third-order non-linear equation and homogeneous Neumann conditions at the boundary. The problem admits a Mean-field type variational formulation, similar to the one obtained by Cruz-Blázquez and Ruiz in \cite{cruzruiz2018} for a related problem in two dimensions, with the associated energy functional being bounded from below but, in general, not coercive. By imposing symmetry conditions, we are able to prove the existence of minimizers, especially when $Q,T\geq 0$. To the best of our knowledge, these are the first existence results obtained for this setting.
\end{abstract}

\section{Introduction}
Given a closed Riemannian surface $(M, g)$, and a smooth function $K$ defined on $M$, the classical Kazdan-Warner problem asks whether $K$ can be realized as the Gaussian curvature of $M$ with respect to a metric in the conformal class of $g$. If we consider a conformal metric of the form $\tilde{g}=e^{2u}g$, then the problem reduces to solving the following Liouville equation on $M$:
\begin{equation}
	\label{eq:gaussian}-\Delta_{g}u+K_g=Ke^{2u}\hsp \text{in } M,
\end{equation}
where $-\Delta_g$ denotes the Laplace-Beltrami operator and $K_g$ is the Gaussian curvature of $(M,g)$. The solvability of \eqref{eq:gaussian} has been studied for a long time, so we refer the interested reader to \cite{Aubin-Variedad_sin_borde} for a detailed overview of the problem. The case of the sphere $\Sph^2$ is known as the Nirenberg's problem and it is especially delicate because of the non-compactness of the Möbius group.

\smallskip 

If the surface has a boundary $\partial M,$ a natural boundary condition to consider is prescribing also the boundary geodesic as a given function $h$. Finding a conformal metric $\tilde{g}=e^{2u}g$ with Gaussian curvature $K$ and boundary geodesic curvature $h$ results equivalent to finding a solution $u$ to the boundary value problem
\begin{align}\label{Kh}
	\left\lbrace \begin{array}{rll}
		-\Delta_{g}u+K_g&=Ke^{2u}, & \mbox{in } M, \\[0.2cm]
	\frac{\partial u}{\partial \nu}+h_g&= he^{u} & \mbox{on } \partial M,
	\end{array}\right. 
\end{align}
where $h_g$ is the original geodesic curvature and $\nu$ denotes the exterior unit normal to $\de M$. Integrating and applying the Gauss-Bonnet theorem, we find a topological invariant linking $K$ and $h$:
\begin{equation*}
	\int_M Ke^{2u}dV_g+\int_{\de M}he^uds_g= \int_{M}K_g\,dV_g+\int_{\de M}h_gds_g=2\pi\chi(M),
\end{equation*}
where $\chi(M)$ denotes the Euler-Poincaré characteristic of $M$. The case of the disk $\mathbb D^2$ can be seen as the analogue of the Nirenberg's problem in this framework, and is again particularly challenging. Via the uniformization theorem, we may assume that $g$ is the Euclidean metric, so the problem \eqref{Kh} becomes

\begin{align}\label{Khdisk}
	\left\lbrace \begin{array}{rll}
		-\Delta u&=Ke^{2u}, & \mbox{in } \mathbb D^2, \\[0.2cm]
		\frac{\partial u}{\partial \nu}+1&= he^{u} & \mbox{on } \partial \mathbb D^2,
	\end{array}\right. 
\end{align}
 under the constraint
 \begin{equation}\label{eq:gbdisk}
 	\int_{\mathbb D^2} Ke^{2u}dV_g + \int_{\de \mathbb D^2}he^{u}ds_g =2\pi,
 \end{equation}
 which forces one of the curvatures to be positive somewhere. For the most recent advances on \eqref{Khdisk}, see \cite{ruiz} and the references therein.
 
 \medskip
 
If $(M,\de M, g)$ is a Riemannian manifold of higher dimension, the geometry becomes richer and one can consider the prescription of different contractions of the curvature tensor. In dimensions $n\geq 3$, one natural analogue to \eqref{Kh} is the prescription of interior scalar curvature and boundary mean curvature $S$ and $H$, respectively, via a conformal change of the metric. For convenience, in this case one searches for a conformal metric $\tilde g=u^{4\over n-2}g$, with $u>0$, which yields to finding a positive solution of 
\begin{align}\label{Yamabe-bdary}
	\left\lbrace \begin{array}{rll}
		-\frac{4(n-1)}{n-2}\Delta_{g}u+S_gu&=Su^{\frac{n+2}{n-2}} & \mbox{in } M, \\[0.2cm]
			\frac{2}{n-2}\frac{\partial u}{\partial \nu}+h_gu&= H u^{n\over n-2} & \mbox{on } \partial M.
	\end{array}\right. 
\end{align}

A general presentation of this problem, known as the boundary Yamabe or Escobar's problem, and a collection of classical results when either $S=0$ or $H=0$ can be found in \cite{BrendleChen}.

\medskip

To explore further conformal and topological properties of curvatures in dimension two, a differential operator $\mathcal P$, closely related to the Laplace-Beltrami operator and mapping functions to 2-forms, was introduced, leading to the definition of the $Q$-curvature (as outlined in \cite{Qcurvature}). In 1991, Branson and {\O}rsted generalized this concept to dimension four, preserving similarities to the scalar curvature properties in dimension two. This was later extended to arbitrary even dimensions, resulting in extensive literature on its construction and properties.

\smallskip

We will now focus on a compact and closed Riemannian manifold $(M,g)$ of dimension $n=4$. The $Q$ curvature introduced by Branson in \cite{branson1985} (see also \cite{BransonOrsted}) is associated to the  \textit{Paneitz operator} $P^4_g$, a differential operator defined by Paneitz in \cite{paneitz1983} which is invariant under conformal transformation. They are defined in terms of the Ricci tensor $\text{Ric}_g$ and the scalar curvature $S_g$ as

\begin{align*}
	P^4_gf &= -\Delta^2_g f + \dv_g\left[\left(\frac{2}{3}S_g g- 2\text{Ric}_g\right)(\nabla f, \cdot)\right]^\sharp, \\
	Q_g &= \frac{-1}{12}\left(\Delta_g S_g-R^2_g+3\vert \text{Ric}_g\vert^2\right),
\end{align*}
for every differentiable function $f$ on $M$, where $^\sharp$ denotes the canonical isomorphism $\sharp:T^*M\to TM$. Similarly to the Laplace Beltrami operator and the Gaussian curvature, the Paneitz operator governs the conformal change of the metric Q. Indeed, under a conformal transformation of the metric given by $\tilde g = e^{2v}g,$ it holds,

\begin{align*}
	P^4_g u + 2Q_g = 2Q_{\tilde{g}}e^{4u} \mbox{ on } M. \label{Q}
\end{align*}

As before, one could ask if, given a function $Q$ defined on $M$, there exists a conformal metric $\tilde g$ such that $Q_{\tilde g}=Q$. The answer is positive when $Q$ is constant and $M$ is not conformally equivalent to the $4-$sphere, see \cite{brendle2003, djadli-malchiodi-Q1, chang-yang-Q}. The analogue of the Nirenberg's problem on $\Sph^4$ is again intrinsically complex and a complete description is still missing. Some partial existence results are available in \cite{wei-xu1998, malchiodi-struwe2006,LiLiLiu} for positive functions $Q$. On the other hand, a Kazdan-Warner type obstruction to the existence of solutions was obtained in \cite{chang-yang-Q}. 

\medskip When the manifold $M$ has a boundary, a differential operator $P_g^3$ defined on $\de M$ was discovered by Chang and Qing in \cite{chang-qing1997}, leading to the $T$ curvature, a third order curvature denoted by $T_g$. They are defined as
\begin{equation*}
\begin{split} 
P^3f &= -\frac12 \frac{\de \Delta_gf}{\de \nu_g}-\Delta_{\hat g}\frac{\de f}{\de\nu_g} +2H_g\Delta_{\hat g}f -\sff_g(\nabla_{\hat g}f,\nabla_{\hat g}f)\\&-\nabla_{\hat g}H_g\nabla_{\hat g}f +\left(\frac 13 S_g-\sum_{a=1}^n{R^a}_{nan}\right)\frac{\de f}{\de \nu_g},
\end{split}
\end{equation*}
and
\begin{equation*}
T_g = \frac1{12}\frac{\de S_g}{\de\nu_g}-\frac12 S_g H_g+\sum_{a,b=1}^{n-1}R_{anbn}(\sff_g)_{ab}-3{H_g}^3+\frac13 \tr (\sff_g)^3-\Delta_{\hat g}H_g,
\end{equation*}
where $\hat g$ stands for the induced metric by $g$ on $\de M$, $\sff_g$ is the second fundamental form of $\de M$ and $R$ is the Riemann curvature tensor. We point out that $(P^3_g,T_g)$ depends not only on the intrinsic geometry of $\de M$, but also on the role of $\de M$ in $M$ (see \cite{chang-qing1997}).

\medskip

On a compact Riemannian 4-manifold with boundary, the conformal transformation of $(Q,T)$ is governed by the couple $(P^4_g,P_g^3)$. Indeed, for the conformal change of metric $\tilde g = e^{2v}g$ we have

\begin{align}
	\left\lbrace \begin{array}{ll}
		P_gv+2Q_g=2Q_{\tilde{g}}e^{4v} & \mbox{in } M \\[0.2cm]
		P^3_g{v}+T_g= T_{\tilde{g}}e^{3v} & \mbox{on } \partial M.
	\end{array}\right. \label{QT}
\end{align}
Therefore, in Conformal Geometry, it is natural to consider the pair $(Q_g,T_g)$ as the generalization of the curvatures $(K_g,h_g)$ in a Riemannian surface. In addition to these resemblances, we have the extension of the Gauss-Bonnet theorem, known as the Gauss-Bonnet-Chern theorem:
\begin{equation}\label{GBC}
\int_M\left(Q_g+\frac{1}{8}\vert W_g\vert^2\right)dV_g+\int_{\partial M}(T_g+Z_g)dS_g = 4\pi^2\chi(M),
\end{equation}
where $W_g$ denotes the Weyl tensor of $(M,g)$, and $Z_g$ is a pointwise conformal invariant operator that vanishes if $\de M$ is totally geodesic.

\medskip Prescribing $Q$ and $T$ curvatures on $M$ and $\de M$ respectively, consists in solving \eqref{QT} with $Q_{\tilde g}=Q$ and $T_{\tilde g}=T$.  This is a fourth order problem with a third order boundary condition, so it is natural to impose an additional first order condition on the boundary. One example that we can find in the literature is to prescribe that the boundary is minimal, that is,
$$\frac{\partial v}{\partial \nu} + H_gu = 0 \hsp \mbox{on } \partial M,$$

Thanks to a result by Escobar \cite{escobar}, it is not restrictive to assume that $H_g=0$, as it can always be achieved through a conformal change of the background metric. 
 
\smallskip
 
 To the best of our knowledge, the prescription of both $Q$ and $T$ curvatures has not been much considered. Existence results for a wide range of manifolds and some compactness results are derived in \cite{chang-qing1997-II,ndiaye2008, ndiaye2011} for the case of constant $Q$ and $T=0$, and in \cite{ndiaye2009} for the prescription of constant $T$-curvature under $Q=0$ (see also \cite{CaseEtAl} for a more recent approach on manifolds with corners). However, none of these results cover the case of a locally conformally flat manifold with totally geodesic boundary and positive integer Euler-Poincaré characteristic. 

\medskip 
In this work, we focus on the latter critical case using the model of the standard upper hemisphere $(\Sph^4_+,\Sph^3,g_0)$, with $\Sph^4_+=\{x\in \R^5:\abs{x}=1,\: x_5\geq0\}$. By direct computation, we have 
\begin{equation*}
	P^4_g=\Delta_g^2-2\Delta_g\end{equation*} and $Q_g=3$. Moreover, the embedding $\de \Sph^4_+=\Sph^3\emb S^4_+$ is totally geodesic, so $\sff_g=0$. This gives $T_g=0$ and
\begin{equation*}
	P^3_gu = -\frac12 \frac{\de \Delta_g u}{\de \nu_g}\end{equation*} for every $u$ such that $\frac{\de u}{\de \nu_g}=0$. Therefore, we are led to solve the following high-order Liouville boundary equation:
\begin{equation}\label{problem}\tag{QT}
	\left\lbrace\begin{array}{rll}
\Delta_g^2 u - 2\Delta_gu + 6 &= 2Qe^{4u} & \text{in}\quad \Sph^4_+, \\[0.5ex]
-\frac{\de}{\de \nu_g}\Delta_gu &= 2T e^{3u}  & \text{on}\quad \Sph^3, \\[0.5ex]
\frac{\de u}{\de \nu_g} &=0 & \text{on}\quad \Sph^3.
	\end{array} \right.
\end{equation}
Furthermore, since $\chi(\Sph^4_+) = \chi(\mathbb{B}^4) = 1,$ \eqref{GBC} becomes
\begin{equation}\label{GBC-sphere}
	\intI Qe^{4u}dV_g + \intB Te^{3u}ds_g = 4\pi^2.
\end{equation}

As hinted by the pairs of equations, \eqref{problem}-\eqref{Khdisk} and \eqref{eq:gbdisk}-\eqref{GBC-sphere}, this problem shares many similarities with the problem of prescribing Gaussian and boundary geodesic curvatures on a disk of $\R^2$, which in turn is related to the Nirenberg problem on the sphere $\Sph^2$. As it happens in the two-dimensional case, the non-compactness of the group of conformal transformations of $\Sph^4_+$ gives rise to \textit{bubbling} of solutions (see \cite{CaseEtAl}).

\smallskip

As a first step in understanding this problem, here we address \eqref{problem} from the point of view of the Calculus of Variations, taking advantage of the Mean-field type formulation found in \cite{cruzruiz2018} for its analogue in two-dimensions. Henceforth, we fix $Q\in C^2(\Sph^4_+)$ and $T\in C^2(\Sph^3)$ to be positive somewhere functions, and we define the parameter $\beta$ formally from \eqref{GBC-sphere} in the following way
\begin{equation*}
	\beta = \intI Qe^{4u}dV_g = 4\pi^2-\intB Te^{3u}ds_g.
\end{equation*}
For every $u$ such that the integrals above are positive, we show in Section \ref{sec:varfor} that \eqref{problem} results equivalent to the following problem for the couple $(u,\beta)$, with $\beta\in(0,4\pi^2)$.

\begin{equation}\label{equivalent}
	\left\lbrace\begin{array}{rll}
		\Delta_g^2 u - 2\Delta_gu + 6 &= 2\beta \dfrac{Qe^{4u}}{\intI Qe^{4u}dV_g} & \text{in}\quad \Sph^4_+, \\[0.5ex]
		-\dfrac{\de}{\de \nu_g}\Delta_gu &= 2(4\pi^2-\beta) \dfrac{T e^{3u}}{\intB Te^{3u}ds_g}  & \text{on}\quad \Sph^3, \\[0.5ex]
		\dfrac{\de u}{\de \nu_g} &=0 & \text{on}\quad \Sph^3, \\[0.5ex]
		\left(\dfrac{4\pi^2-\beta}{\intB Te^{3u}ds_g}\right)^{4/3} &= \dfrac{\beta}{\intI Qe^{4u}dV_g} & \text{in}\quad (0,4\pi^2).
	\end{array} \right.
\end{equation}

Problem \eqref{equivalent} can be seen as the Euler-Lagrange equation of the energy functional 
\begin{equation*}
	\begin{split}
		I(u,\beta) &= \intI (\Delta_gu)^2+2\intI\abs{\nabla_g u}^2+12\intI u-\beta \log\intI Qe^{4u} \\&-\frac43 (4\pi^2-\beta)\log \intB Te^{3u} + \beta\log\beta +\frac43(4\pi^2-\beta)\log(4\pi^2-\beta)+\frac\beta3,
	\end{split}
\end{equation*} 
for which we can provide energy estimates using higher-order Moser-Trudinger type inequalities. 

\medskip

Indeed, one of the main goals of this work is to prove sharp inequalities for the case of the upper hemisphere. These inequalities are available for $\Sph^4$ in \cite{beckner}, but their adaptation to the boundary case requires specific techniques. In this regard, we emphasize that the direct application of the results obtained in \cite{ndiaye2008,ndiaye2009} for general manifolds with boundary is not sufficient in our case, due to the loss of a small $\e>0$ (see Section \ref{sec:MT}).

\medskip

The use of the aforementioned inequalities yields that the Euler-Lagrange functional is uniformly bounded from below, but as it happens in the two-dimensional case studied in \cite{cruzruiz2018}, it is not certain that a global minimizer exists. In the spirit of \cite{moser73}, we should impose symmetry assumptions to gain coercivity.

\medskip

More precisely, we let $\mathcal G$ be a group of symmetries of $\Sph^4_+$, and let $\mathcal F_{\de}$ denote the set of boundary points fixed by the action of $\mathcal G$, that is, 
\begin{equation}\label{fixed-points}
\fixed = \left\lbrace x\in \Sph^3: \varphi(x)=x\hsp\forall \varphi\in \mathcal G \right\rbrace.
\end{equation}

We are particularly interested in symmetry groups such that $\mathcal F_\de=\emptyset$ or $\mathcal F_\de=\Sph^k$, with $k=0,1,2$, up to an isometry of $\R^5$. This condition will be assumed for the rest of the paper.

\smallskip

  Examples of these groups are, for instance, ciclic groups generated by a rotation of two coordinates around the $x_5-$axis with minimal angle ${2\pi\over k}$, with $k\in \mathbb Z$ and $k\geq 2$ ($\mathcal F_\de \simeq \Sph^1$), the product of two of these groups ($\mathcal F_\de = \emptyset$), groups generated by a reflection with respect to a hyperplane that leaves $x_5$ invariant ($\mathcal F_\de \simeq \Sph^2$), or product groups of the form $\text{Id}\times O(3)$ ($\mathcal F_\de \simeq \Sph^0$).

\medskip

Our main result is the following:

\begin{theorem}\label{maintheorem} Let $Q\in C^2(\Sph^4_+)$ and $T\in C^2(\Sph^3)$ be non-negative and $\mathcal G$-symmetric functions, not both of them identically equal to zero. Assume that $T=Q=0$ along $\mathcal F_\de $. Then, problem \eqref{problem} admits a solution.
\end{theorem}

The idea of the proof goes as follows: let $(u_n,\beta_n)$ be a minimizing sequence for $I$, where $u_n$ are $\mathcal G-$symmetric functions in a suitable subspace of $H^2(\Sph^4_+)$, and $\beta_n\in(0,4\pi^2)$. For every fixed $\beta$, we show that the functional $u\to I(u,\beta)$ admits a global minimizer $u_\beta$. This is done by studying the possible regions into which the masses of the sequences $Qe^{4u_n}$ and $Te^{3u_n}$ are distributed. By symmetry, the regions away from $\mathcal F_\de$ with mass appear in duplicate, allowing us to apply improved versions of the Moser-Trudinger inequalities, known as Chen-Li type inequalities (see \cite{chen-li91}), and so gaining coercivity. If $\mathcal F_\de \neq \emptyset$, a minimizing sequence could concentrate around the fixed points, preventing the splitting of mass. This possibility is dismissed with the additional assumption that $Q=T=0$ on $\mathcal F_\de$, as done in \cite{liu-xu} to prescribe a geodesic curvature on the boundary of a disk that is symmetric under a reflection.

\medskip 

The possibility $\mathcal F_\de\neq \emptyset$ makes the study more delicate, and it becomes necessary to provide localized versions of the sharp Moser-Trudinger type inequalities that are more suitable to work with the weights $Q$ and $T$.

\medskip

To conclude, we show $\beta_n\to\beta\in (0,4\pi^2)$ via energy estimates. This argument requires a detailed study of the limiting cases $\beta=0$ and $\beta=4\pi^2$, corresponding with prescribing $T\geq 0$ with $Q=0$ and viceversa. In these cases, we are able to remove the non-negativity assumption when $\mathcal F_\de = \emptyset$. This is carried out in Section \ref{sec:lc}.

\medskip

We remark that, when $\mathcal F_\de = \emptyset$, our results can be seen as counterparts of the ones found in \cite{cruzruiz2018} for the Gaussian-geodesic prescription problem. On this subject, it is important to note that the adaptation of some of the methods used in that work required techniques that are specific to dimension four, and which can not be directly extended to the more complex case $\mathcal F_\de \neq \emptyset$.

\medskip

We finally point out the relation between \eqref{QT} and the boundary Yamabe problem \eqref{Yamabe-bdary}. 
The definition of a one-parameter family of operators, called conformal fractional Laplacians $P^s_g,$ (see \cite{GZ}) gave raise to a family $Q_s$ of intrinsic curvatures with good conformal properties. Here $s\in(0,\frac{n}{2})$, and $s\in N$ needs to be treated separately. In general, $Q_s$ with $s\notin N$ are non-local quantities, but $Q_1$ and $Q_2$ are, up to normalization constant, the scalar and $Q-$ curvature \cite{ChangQ}. The prescription of these curvatures is known as fractional Yamabe problem. The original definition of the conformal fractional Laplacian is related to a Dirichlet-to-Neumann operator (see \cite{ChangGonzalez}), thus the PDE formulation can be treated through an elliptic extension problem of the type \eqref{Yamabe-bdary} with $S=0$ (see also \cite{CaseChang} or \cite{Case} for higher order extensions). Therefore, the boundary equation \eqref{QT} with $\tilde Q_{\tilde{g}}\equiv 0$ can be seen as the limiting case (as $s\to\frac{n}{2}$) of the prescribed $Q_{\bar g}^s=T$ curvature equation (see \cite[Section 6]{DGHM}).
\section*{Acknowledgements}
S.C. acknowledges financial support from the Spanish Ministry of Universities and Next Generation EU funds, through a \textit{Margarita Salas} grant from the University of Granada and by the FEDER-MINECO Grant PID2021-122122NB-I00. 

\smallskip

A. DlT. acknowledges financial support from the Spanish Ministry of Science and Innovation (MICINN), through the IMAG-Maria de Maeztu Excellence Grant CEX2020-001105-M/AEI/ 10.13039/501100011033 and FEDER-MINECO Grants PID2021- 122122NB-I00 and PID2020-113596GB- I00; RED2022-134784-T, funded by MCIN/AEI/10.13039/501100011033. She is also supported by Fondi Ateneo – Sapienza Università di Roma and PRIN (Prot. 20227HX33Z).  

\smallskip

Both authors have been partially supported by the IN$\delta$AM-GNAMPA projects 2023 CUP E53C2200193000 and 2024 CUP E53C23001670001. They both acknowledge financial support from J. Andalucia (FQM-116). Part of this work was carried out during A. DlT.'s visits to the University of Granada and S.C's visit to the Sapienza Università di Roma, to which both authors are grateful. 
\section{Notation and preliminaries}
In this section we establish the notation we will be using throughout the paper, and state some technical results we will refer to. 

\medskip

\subsection{Notation} Given a subset $A$ of a metric space $(X,d)$ and $t>0$, $(A)^t$ will denote the tubular neighborhood 
\begin{equation*}
	(A)^t = \left\lbrace x\in X:\: d(x,A)<t\right\rbrace.
\end{equation*} 

In the case of a $4$-dimensional compact manifold with boundary $(M,\de M,g)$, given $\rho>0$ and $p\in M$, we denote by $B^4_\rho(p)$ the geodesic ball centered at $p$ of radius $\rho$. Instead, if $p\in \de M$, we will use $B^3_\rho(p)$ to denote the corresponding ball in $\de M$ with respect to the induced distance. 

\medskip

For the sake of simplicity, when we work on $\Sph^4_+$ with its standard metric, we drop the volume or area element in our integrals, as well as the sub-indexes depicting the dependence on the metric $g$. For instance, we simply write

\begin{equation*}
	\intI (\Delta_g u)^2dV_g = \intI (\Delta u)^2.
\end{equation*}

\smallskip

Moreover, in our computations, $C$ will be used to denote a positive constant that may change from line to line, or even within the same line. Moreover, arbitrarily small constants $\e>0$ will be relabelled with no specific mention. 

\smallskip

Finally, it is helpful to remember that $\text{Vol}^4(\Sph^4_+)=\frac{4\pi^2}{3}$ and $\text{Vol}^3(\Sph^3)=2\pi^2$.

\subsection{Functional spaces} Let $(M,\de M,g)$ be a $4-$dimensional compact manifold with boundary. We denote
\begin{equation*}
	H^2_\de(M)=\left\lbrace u\in H^2(M):\:\frac{\de u}{\de\nu_g}=0\hsp\text{on}\hsp \de M\right\rbrace.
\end{equation*}
Similarly to \cite{ndiaye2008,ndiaye2009,ndiaye2011}, we define the operator $P^{4,3}_g$ through
\begin{equation*}
	\begin{split}
		\scal{P^{4,3}_gu,v}_{L^2(M)} &= \int_M\Delta_gu\Delta_g v\:dV_g + \frac23\int_M S_g \nabla_gu\nabla_g v\:dV_g \\&-2\int_M \text{Ric}_g\left(\nabla_gu,\nabla_gv\right)\:dV_g -2\int_{\de M}\sff_g\left(\nabla_{\hat g}u\nabla_{\hat g}v\right)\:ds_g,
	\end{split}
\end{equation*}
for $u,v\in H^2_\de(M)$. Under certain conditions, the following result ensures that it induces an equivalent norm to the standard one in $H^2_\de(M)$.

\begin{lemma}[{{\cite[Lemma 2.9]{ndiaye2008}}}]\label{norm-equiv} Suppose that $P^{4,3}_g$ is a non-negative operator with \newline $\ker P^{4,3}_g\simeq \R$. Then, the map
\begin{equation*}
u\mapsto  \(\scal{P^{4,3}_gu,u}_{L^2(M)}\)^{\frac12}
\end{equation*}
	is an equivalent norm to $\norm{\cdot}_{H^2(M)}$ on $H^2_\de(M)$.
\end{lemma}

\smallskip

In the case $M=\Sph^4_+$, the operator $P^{4,3}$ reduces to the simpler expression
\begin{equation*}
	\scal{P^{4,3}u,u}_{L^2(\Sph^4_+)} = \intI \left(\Delta u\right)^2 + 2\intI \abs{\nabla u}^2.
\end{equation*} 
Clearly, $P^{4,3}$ is a nonnegative operator on $H^2_\de(\Sph^4_+)$ and $\ker P^{4,3}\simeq \R$. 

\medskip

Now, we introduce the natural spaces of definition for our energy functionals. Let $Q\in C^2(\Sph^4_+)$ and $T\in C^2(\Sph^3)$ be positive somewhere functions. We set
\begin{align*}
	\HH &= \left\lbrace u\in H^2_\de(\Sph^4_+):\intI Qe^{4u}>0,\:\intB Te^{3u}>0\right\rbrace, \\
	\HH^0 &= \left\lbrace u\in H^2_\de(\Sph^4_+):\intB Te^{3u}>0\right\rbrace, \quad\text{and}\\
	\HH^{4\pi^2} &= \left\lbrace u\in H^2_\de(\Sph^4_+):\intI Qe^{4u}>0\right\rbrace.
\end{align*}
Similarly to \cite[Lemma 2.1]{cruzruiz2018}, one can show that these sets are non-empty. Moreover, observe that, if $Q$ and $T$ are non-negative and positive somewhere, we have $\mathcal H = \mathcal H^0=\mathcal H^{4\pi^2} = H^2_\de(\Sph^4_+)$. We will also consider the Hilbert subspace of $H^2_\de(\Sph^4_+)$ formed by the $\mathcal G-$symmetric functions, 
\begin{equation*}
	H^2_{\mathcal G}(\Sph^4_+) = \bigg\lbrace u\in H^2_\de(\Sph^4_+):  u(\varphi) = u \quad \forall \varphi \in \mathcal G \bigg\rbrace,
\end{equation*}
and denote by $\HH_{\mathcal G}$, $\HH^0_{\mathcal G}$ and $\HH^{4\pi^2}_{\mathcal G}$ the corresponding intersections of the spaces defined above with $H^2_{\mathcal G}(\Sph^4_+)$. 

\medskip

\subsection{A useful covering of $\overline{\Sph^4_+}$} Next, we construct a suitable covering of $\Sph^4_+$ and its boundary that will allow us to control the distribution of the mass of $\mathcal G$-symmetric functions.

\medskip 
Recall that $\mathcal F_{\de}$ denote the set of boundary points fixed by the action of $\mathcal G$ as stated in \eqref{fixed-points}.
First, if $\mathcal F_\de\neq \emptyset$, we fix a small enough $\delta>0$ and consider the geodesic neighborhood of $\mathcal F_\de$ of radius $\delta$ on $\Sph^3$, that is,
\begin{equation}\label{eq:neig}
	U_{\delta}=\{x\in\Sph^3: \text{dist}(x,\mathcal F_\de)<\delta\}.
\end{equation}
Instead, if $\mathcal F_\de = \emptyset$, we just take $U_\delta=\emptyset$. In both cases, we clearly have $U_\delta\neq \Sph^3$ because of our assumption on $\mathcal G$. Now, we restrict ourselves to the compact set $\Sph^3\setminus U_{\delta/2}$. By compactness, there exists $\rho>0$ and a finite collection of points $\{p_1,\ldots,p_N\}$ such that $\Sph^3\setminus U_{\delta/2}\subset \bigcup_{i=1}^NB^3_\rho(p_i).$ Finally, for every $i=1,\ldots,N$, we have $p_i\notin\mathcal F_\de$, so there exists an element $\varphi_i\in \mathcal G$ such that $\varphi_i(p_i)\neq p_i.$

\begin{lemma}\label{lem:cov} Let us call $A_i=B^3_{\rho}(p_i)$ for $i=1,\ldots,N$. The set $$\mathcal A=\left\lbrace U_{\delta}, A_1, \ldots, A_N, \varphi_1(A_1),\ldots,\varphi_N(A_N)\right\rbrace$$ constructed above is a finite open covering of $\Sph^3$. Moreover, if $\rho>0$ is small enough, there exists $\e>0$ such that
	$$\text{dist}(A_i, \varphi_i(A_i))>\e \hsp\text{for all }i=1,\ldots,N.$$
\end{lemma}
\begin{proof}
	$\mathcal A$ is a finite covering of $\Sph^3$ by construction. Now, fix $p_i\in\Sph^3\setminus U_{\delta/2}$. Since $p_i\notin\mathcal F_\de $, there exists $\varphi_i\in \mathcal G$ such that $\varphi_i(p_i)\neq p_i$. By the continuity of $\varphi_i$, this property holds true in a small neighborhood of $p_i$ of radius $r_i>0$. Therefore, it is sufficient to take $\rho<\min \left\lbrace r_i: i=1,\ldots,N \right\rbrace.$
\end{proof}
Next, we extend this construction to a finite covering of $\Sph^4_+$ using tubular neighborhoods. We fix a small enough $t>0$ and consider, for every $B\subset \mathcal \Sph^3$, the subset of $\Sph^4_+$ given by
\begin{equation*}
	(B)^t = \{x\in \Sph^4_+:\dist(x,B)<t\}.
\end{equation*}
Reasoning as before, the following is a finite covering of $\overline{(\Sph^3)^{t/2}}$ in $\Sph^4_+:$
\begin{equation*}
	\mathcal B = \left\lbrace(U_\delta)^t, (A_1)^t,\ldots, (A_N)^t, \varphi_1((A_1)^t),\ldots,\varphi_N((A_N)^t)\right\rbrace.
\end{equation*}
Thus, calling $\Omega=\Sph^4_+\setminus \bigcup_{B\in\mathcal B}B,$ we have the following property:
\begin{lemma}\label{lem:cov2} The set $\mathcal B \cup \{\Omega\}$, with $\mathcal B$ as described above, is a finite covering of $\Sph^4_+$. Moreover, if $\rho$ and $t$ are small enough, there exists $\e>0$ such that
	\begin{equation*}
		\begin{split}
			\dist\left((A_i)^t,\varphi_i((A_i)^t)\right)&>\e\hsp\text{for all } i=1,\ldots,N, \\
			\dist(\Omega,\de \Sph^4_+) &>\e.
		\end{split}
	\end{equation*} 
\end{lemma}

The main purpose of the previous covering is the following alternative for the distribution of the masses of $\mathcal G-$symmetric functions.
\begin{proposition}\label{pr:cases}
	Let $Q$ and $T$ be non-negative functions that are positive somewhere, and $u_n$ be a sequence in $H^2_{\mathcal G}(\Sph^4_+)$. Up to taking a subsequence, there exists $\gamma\in (0,1)$ such that one of the following holds true:
	\begin{enumerate} 
		\item[(i)]
		\begin{equation*}\int_{(U_\delta)^t}Qe^{4u_n}\geq \gamma\intI Qe^{4u_n}, \hsp\text{or}
		\end{equation*}
		\item[(ii)] there exists $i\in\{1,\cdots, N\}$ with $(A_i)^t\in\mathcal B$ such that 
		\begin{equation*}\int_{(A_i)^t}Qe^{4u_n}=\int_{\varphi_i\((A_i\)^t)}Qe^{4u_n}\geq \gamma\intI Qe^{4u_n}, \hsp\text{or}
		\end{equation*}
		\item[(iii)] 
		\begin{equation*}
			\int_{\Omega}Qe^{4u_n}\geq \gamma\intI Qe^{4u_n}.
		\end{equation*}
	\end{enumerate}
	Moreover, there exists $\lambda\in (0,1)$ such that it holds
	\begin{enumerate} 
		\item[(a)]
		\begin{equation*}\int_{U_\delta}Te^{3u_n}\geq \lambda\intB Te^{3u_n}, \hsp\text{or}
		\end{equation*}
		\item[(b)] there exists $i\in\{1,\cdots, N\}$ with $A_i\in\mathcal A$ such that 
		\begin{equation*}\int_{A_i} Te^{3u_n}=\int_{\varphi_i(A_i)}Te^{3u_n}\geq \lambda \intB Te^{3u_n}.
		\end{equation*}
	\end{enumerate}
\end{proposition}
\begin{proof}
	Suppose by contradiction that, for each $\gamma\in (0,1)$, there exists $m\in\N$ such that, for every $n\geq m$, one has
	\begin{equation*}
		\begin{split} 
			\int_{(U_\delta)^t}Qe^{4u_n}&<\gamma\intI Qe^{4u_n},\\ \int_{(A_i)^t}Qe^{4u_n}&<\gamma \intI Qe^{4u_n}\hsp \forall i=1,\ldots,N, \quad \text{and}\\
			\int_{\Omega}Qe^{4u_n}&< \gamma\intI Qe^{4u_n}.
		\end{split}	
	\end{equation*}
	We can sum the inequalities above to obtain
	\begin{equation*}
		\int_{\Omega}Qe^{4u_n} +\int_{(U_\delta)^t}Qe^{4u_n}+\sum_{i=1}^N \int_{(A_i)^t}Qe^{4u_n} < \left(2+N\right)\gamma\intI Qe^{4u_n}.
	\end{equation*}
	Now, since $\Sph^4_+ \subset \Omega\cup(U_\delta)^t\cup \bigcup_{i=1}^N (A_i)^t$ and $Q\geq 0$, one has
	\begin{align*}
		0<\intB Qe^{4u_n} &\leq 	\int_{\Omega}Qe^{4u_n} +\int_{(U_\delta)^t}Qe^{4u_n}+\sum_{i=1}^N \int_{(A_i)^t}Qe^{4u_n} \\ &< (2+N)\gamma\intI Qe^{4u_n},
	\end{align*}
	and it is enough to take $\gamma < (2+N)^{-1}$ to reach a contradiction. The second alternative follows from a completely analogous reasoning.
\end{proof}

\section{Variational formulation of the problem}\label{sec:varfor}
In this section we follow \cite[\textsection 2.2]{cruzruiz2018} to rewrite \eqref{problem} in the form of a Mean-field type equation, and derive the corresponding Euler-Lagrange functional. 

\begin{proposition} Assume $Q$ and $T$ are positive somewhere functions, and take $u\in\HH$ and $\beta\in(0,4\pi^2)$. Then, problems \eqref{problem} and \eqref{equivalent} are equivalent.
\end{proposition}

\begin{proof} Let $u$ be a solution of \eqref{problem}. Then, by \eqref{GBC-sphere}, $\left(u,\intI Qe^{4u}dV_g\right)$ is a solution of \eqref{equivalent}. Conversely, if $(u,\beta)$ solves \eqref{equivalent}, it is possible to find a solution of \eqref{problem} of the form $u+C$, with $C\in\R.$ Indeed, the first two equations read as
\begin{equation*}
\begin{split}
\Delta_g^2 u - 2\Delta_gu + 6 &=  \dfrac{\beta e^{4C}}{\intI Qe^{4u}dV_g}2Qe^{4u}, \\
-\frac{\de}{\de \nu_g}\Delta_gu &=  \dfrac{(4\pi^2-\beta)e^{3C}}{\intB Te^{3u}ds_g}2T e^{3u}.
\end{split}
\end{equation*}
To recover \eqref{problem}, we need
\begin{equation*}
e^{C}=\left(\dfrac{\intI Qe^{4u}dV_g}{\beta}\right)^{1/4} = \left(\dfrac{\intB Te^{3u}ds_g}{4\pi^2-\beta}\right)^{1/3}.
\end{equation*}
The fourth equation of \eqref{equivalent} guarantees that the condition established by the second equality above always holds true. Hence, it suffices to take \newline $C=\frac14\log \left(\frac1\beta \intI Qe^{4u}dV_g\right)$.
\end{proof}

\begin{proposition}\label{euler-lagrange} Weak solutions of \eqref{problem} can be obtained as critical points of the energy functional
	\begin{equation}\label{functional} 
		\begin{split}
		I(u,\beta) &= \intI (\Delta_gu)^2+2\intI\abs{\nabla_g u}^2+12\intI u-\beta \log\intI Qe^{4u} \\&-\frac43 (4\pi^2-\beta)\log \intB Te^{3u} +\mathfrak f(\beta),
		\end{split}
	\end{equation} 
defined for $\beta\in [0,4\pi^2]$ and  $u\in \HH$ if $\beta\in (0,4\pi^2)$, $u\in \HH^0$ if $\beta = 0$ or $u\in \HH^{4\pi^2}$ if $\beta=4\pi^2$. Here, $\mathfrak f$ is the continuous extension to $[0,4\pi^2]$ of the function 
\begin{equation}\label{def:f}
\beta\mapsto \beta\log\beta +\frac43(4\pi^2-\beta)\log(4\pi^2-\beta)+\frac\beta3.
\end{equation} \qed
\end{proposition}
 
For convenience, given $\beta \in [0,4\pi^2]$, we call $I_\beta$ the functional $u\mapsto I_\beta(u)=I(u,\beta)$. In particular,
\begin{align}
I_0(u) &= \intI (\Delta_gu)^2+2\intI\abs{\nabla_g u}^2+12\intI u -\frac{16\pi^2}{3} \log \intB Te^{3u} +\frac{16\pi^2}{3}\log(4\pi^2), \label{functional0}\\
	I_{4\pi^2}(u) &= \intI (\Delta_gu)^2+2\intI\abs{\nabla_g u}^2+12\intI u -4\pi^2 \intI Qe^{4u} + 4\pi^2\log(4\pi^2)+\frac{4\pi^2}{3}, \label{functional4pi2}
\end{align}
defined for every $u$ in $\HH^0$ and $\HH^{4\pi^2}$ respectively.

\medskip 

We conclude this section with a useful property of the energy functional \eqref{functional}, which will allow us to only consider functions with zero mean value.
\begin{proposition}\label{pr:addition-constants}
For every $\beta\in [0,4\pi^2]$, the functional $I_\beta$ is invariant under the addition of constants.
\end{proposition}

\begin{proof}
By direct computation,
\begin{align*}
I(u+C,\beta) &= I(u,\beta)+12C\cdot\frac{8\pi^2}{6}-4\beta C - \frac43 (4\pi^2-\beta)3C \\ &=I(u,\beta)+16\pi^2C - 4\beta C - 16\pi^2 C+4\beta C = I(u,\beta). 
\end{align*}
\end{proof}
\section{Sharp higher-order Moser-Trudinger inequalities}\label{sec:MT}
This section is devoted to proving some higher order Moser-Trudinger type inequalities for functions in $H^2_\de(\Sph^4_+)$, from which we will show that the energy is uniformly bounded from below. However, these inequalities are sharp for the case of the half-sphere and it is not possible to attain coercivity. To deal with this issue, we will derive localized versions that will help us to refine the previous results when working with functions whose mass is distributed in several separate regions, as in the case of symmetric functions, in the spirit of \cite{aubin79} (see also \cite{chen-li91}).

\begin{proposition}\label{Prop:MT1} There exists a constant $C$ such that 
	\begin{equation}\label{MT1}
		\log \intI e^{4(u-\bar u_{\Sph^4_+})} \leq C + \frac{1}{4\pi^2}\scal{P^{4,3}u,u}_{L^2(\Sph^4_+)},
	\end{equation}
	for every $u\in H^2(\Sph^4_+)$ with $\frac{\de u}{\de \eta} =0$.
\end{proposition}
\begin{proof} The analogue inequality to \eqref{MT1} in the closed sphere $\Sph^4$ is due to Beckner \cite{beckner} (see also \cite[Eq. (2.20)]{chang-qing1997-II}):
\begin{equation}\label{becknerInt} \log \int_{\Sph^4}e^{4(u-\bar u_{\Sph^4})} \leq C+\frac{1}{8\pi^2}\int_{\Sph^4}\((\Delta u)^2+2\abs{\nabla u}^2\).
\end{equation}
The proof follows the well-known strategy of extending any function $u\in H^2(\Sph^4_+)$ to $\Sph^4$ by symmetry, possible due to the condition $\frac{\de u}{\de \eta}=0$, applying the inequality \eqref{becknerInt} and returning to the upper hemisphere by a change of variables. For $u\in H^2(\Sph^4_+)$ with $\frac{\de u}{\de \eta}=0,$ we define its symmetrization to $\Sph^4$ as the function $u^\star:\Sph^4\to \R$ given by
\begin{equation*}
	u^\star(x)=\left\lbrace \begin{array}{cc}
u(x) & \text{if}\hsp x_5\geq0, \\
u(R(x)) & \text{if}\hsp x_5<0,
	\end{array} \right.
\end{equation*}
where $x=(x_1,x_2,x_3,x_4,x_5)$ and $R(x)=(x_1,x_2,x_3,x_4,-x_5)$. Integrating by parts and using the condition $\frac{\de u}{\de x_5}=0$, it can be shown that the functions $g,h:\Sph^4_+\to \R$ given by the formulae
\begin{equation*} 
	g(x)=\left\lbrace \begin{array}{cc}
		\frac{\de u(x)}{\de x_5} & \text{if}\hsp x_5\geq0, \\[1ex]
		-\frac{\de u(R(x))}{\de x_5} & \text{if}\hsp x_5<0,
	\end{array} \right. \quad \text{and} \quad h(x)=\left\lbrace \begin{array}{cc}
	\frac{\de^2 u(x)}{\de x_5^2} & \text{if}\hsp x_5\geq0, \\[1ex]
	\frac{\de^2 u(R(x))}{\de x_5^2} & \text{if}\hsp x_5<0,
	\end{array} \right.
\end{equation*}
are the first and second weak derivatives of $u^\star$ with respect to $x_5$, respectively, and 
\begin{equation*}
\norm{u^\star}^2_{H^2(\Sph^4)}=2\norm{u}_{H^2(\Sph^4_+)}^2.
\end{equation*}
Therefore, we can apply \eqref{becknerInt} to $u^\star$, obtaining
\begin{align}\label{eq*}
\log \int_{\Sph^4}e^{4(u^\star-\overline{u^\star}_{\Sph^4})}\leq C+\frac{1}{8\pi^2}\int_{\Sph^4}\left((\Delta_g u^\star)^2+2\abs{\nabla_g u^\star}^2\right).
\end{align}
By means of the change of variables $y=R(x)$, we immediately see that $\overline{u^\star}_{\Sph^4} = \overline{u}_{\Sph^4_+}$, which implies
\begin{equation}\label{eq7-2}
\log\int_{\Sph^4}e^{4(u^\star-\overline{u^\star}_{\Sph^4})} = \log \int_{\Sph^4_+}e^{4(u-\bar u_{\Sph^4_+})}+\log 2.
\end{equation}
Furthermore,
\begin{equation}\label{eq7-3}
\int_{\Sph^4}\left((\Delta_g u^\star)^2+2\abs{\nabla_g u^\star}^2\right) = 2\int_{\Sph^4_+}\left((\Delta_g u)^2+2\abs{\nabla_g u}^2\right).
\end{equation}
Equation \eqref{MT1} follows from using \eqref{eq*} and \eqref{eq7-3} in \eqref{eq7-2}.
\end{proof}
Next, we present a boundary version of the sharp inequality \eqref{MT1}. It is a particular case of \cite[Theorem 6.14]{Case2} when $u$ satisfies $\frac{\de u}{\de \eta_g}=0$.
\begin{proposition} There exists a constant $C$ such that
	\begin{equation}\label{MT2}
		\log \intB e^{3(u-\bar u_{\Sph^3})}\leq C+\frac{3}{16\pi^2}\scal{P^{4,3}{u},u}_{L^2(\Sph^4_+)},
	\end{equation}
	for every $u\in H^2_\de(\Sph^4_+)$.
\end{proposition} 

In order to interpolate these inequalities, we give a version of \eqref{MT2} with the mean value of $u$ in the interior of $\Sph^4_+$. 
\begin{proposition}\label{prop:MT2-bis}There exists a constant $C$ such that
	\begin{equation}
 \label{MT2-bis}\log \intB e^{3(u-\bar u_{\Sph^4_+})}\leq C+\frac{3}{16\pi^2}\scal{P^{4,3}u,u}_{L^2(\Sph^4_+)},
	\end{equation}
	for every $u\in H^2_\de(\Sph^4_+)$.
\end{proposition}
\begin{proof}

\medskip

First, we show that the problem
\begin{equation}\label{aux-prob}
\left\lbrace\begin{array}{ll}
\Delta_g^2w-2\Delta_gw = 6 & \text{in } \Sph^4_+, \\[0.5ex]
-\frac{\de}{\de \nu_g} \Delta_g w =  -4 & \text{on } \Sph^3, \\[0.5ex]
\frac{\de w}{\de \nu_g} = 0 & \text{on } \Sph^3.
\end{array} \right.
\end{equation}
admits a solution. This can be done by means of the direct method of the Calculus of Variations, noticing that \eqref{aux-prob} is the Euler-Lagrange equation associated to the functional
\begin{equation*}
J(w)=\frac12 \scal{P^{4,3}w,w}_{L^2(\Sph^4_+)}-6\intI w+4\intB w,
\end{equation*}
defined for every $w\in H^2_\de(\Sph^4_+)$. By Sobolev and trace inequalities, there exists a positive constant $C$ such that
\begin{equation}\label{wololo1}
-6\intI w+4\intB w \geq -C \norm{w}_{H^2(\Sph^4_+)}.
\end{equation}
Finally, by \eqref{wololo1} and Lemma \ref{norm-equiv}, we get 

$$J(w)\geq \tilde C\norm{w}^2_{H^2(\Sph^4_+)}-C\norm{w}_{H^2(\Sph^4_+)},$$
 where $\tilde C$ is also a positive constant. Consequently, $J$ is coercive in its space of definition, which leads towards the existence of a global minimizer following standard methods.

\medskip

Once this has been established, we fix a solution $w$ of \eqref{aux-prob} and apply \eqref{MT2} to $u+w$, obtaining
\begin{equation*}
\begin{split}
\log \intB e^{3u} &\leq C +\frac{3}{16\pi^2}\intI\left((\Delta_gu)^2+2\abs{\nabla_gu}^2\right)\\ &+ \frac{3}{8\pi^2} \intI \left(\Delta_gu \Delta_g w+2\nabla_g u \nabla_g w\right)+\frac{3}{2\pi^2}\intB u.
\end{split}
\end{equation*}
Integrating by parts and using \eqref{aux-prob}, we obtain
\begin{equation*}
\begin{split}
\log \intB e^{3u} &\leq C + \frac{3}{16\pi^2}\intI\left((\Delta_gu)^2+2\abs{\nabla_gu}^2\right)+\frac{3}{8\pi^2}\intI u\left(\Delta_g^2 w-2\Delta_g w\right) \\ &+\intB u \left(-\frac{3}{8\pi^2}\frac{\de }{\de \nu_g}\Delta_g w+\frac{3}{4\pi^2}\frac{\de w}{\de \nu_g}+\frac{3}{2\pi^2}\right) \\ &= C + \frac{3}{16\pi^2}\intI\left((\Delta_gu)^2+2\abs{\nabla_gu}^2\right) + \frac{9}{4\pi^2}\intI u.
\end{split}
\end{equation*}
To conclude, it suffices to notice that $\text{Vol}^4(\Sph^4_+)=\frac{4\pi^2}{3}$.
\end{proof}
Now we present localized versions of \eqref{MT1} and \eqref{MT2-bis} that will be of particular usefulness when we can guarantee that the mass of the function $e^u$ is concentrated in the interior of $\Sph^4_+$ or distributed in separated regions along its boundary.

\begin{proposition}\label{local-int1} Let $\Omega\subset \Sph^4_+$ be such that $\dist(\Omega,\Sph^3)>\delta$ for some $\delta>0$. Then, for every $\e>0$, there exist a constant $0<C=C(\Omega,\e,\delta)$ and a large enough eigenvalue $\lambda=\lambda(\e,\delta)$ of $P^{4,3}$, such that 
	
	\begin{equation}\label{local-ineq-1}
		\begin{split}
			\log \int_{\Omega} e^{4u} \leq C + \frac{1}{8\pi^2}\int_{\Omega^\delta}\((\Delta u^\dagger)^2+2\abs{\nabla u^\dagger}^2\)+\e\scal{P^{4,3}u,u}_{L^2(\Sph^4_+)}
		\end{split}
	\end{equation}
	for every $u\in H^2(\Sph^4_+)$ with $\frac{\de u}{\de\eta_g}=0,$ where $u^\dagger$ denotes the projection of $u$ to $E_\lambda^\perp$, being $E_\lambda$ the direct sum of the eigenspaces of $P^{4,3}$ with eigenvalues smaller or equal than $\lambda.$
\end{proposition}
\begin{proof} The proof follows some of the ideas introduced in \cite[Lemma 2.2]{djadli-malchiodi-Q1} and their adaptation to the boundary case in \cite[Lemma 4.1]{ndiaye2008}. In our case, when $\mathcal F_\de \neq \emptyset$, it is not convenient to remove the weights $Q$ and $T$ from the integral terms, and the direct application of the aforementioned results is not enough to conclude. Therefore, we must introduce substantial changes in order to get localized versions.
	
\smallskip 

First, we consider a smooth cut-off function $\chi$ satisfying
\begin{equation*}
\left\lbrace \begin{array}{ll}
\chi(x) \in [0,1] & \text{for every} \hsp x\in\Sph^4_+,\\
\chi(x) = 1 & \text{if} \hsp x\in \Omega, \\
\chi(x) = 0 & \text{if} \hsp \dist\(x,\Omega\)\geq \frac\delta 2, \\
\norm{\chi}_{C^2(\Sph^4_+)} \leq 1.
\end{array}\right.
\end{equation*}

Let $u\in H^2(\Sph^4_+)$, and assume without loss of generality that $\overline{u}{_{\Sph^4_+}}=0$. We call $v$ the extension by zero of $\chi u$ to $\Sph^4$, which belongs to $H^2(\Sph^4)$ due to the properties of $\Omega$. We write $u=u_1+u_2$, with $u_1\in L^\infty(\Sph^4)$, and define $v_1=\chi u_1$ and $v_2=\chi u_2$. 

\smallskip 

We can estimate as follows:

\begin{equation*}
\begin{split}
\log \int_{\Omega} e^{4u} &\leq \log \int_{\Sph^4_+} e^{4\chi u} \leq \log \int_{\Sph^4} e^{4v} \leq  \log\left( e^{4\norm{v_1}_{L^\infty(\Sph^4)}}\int_{\Sph^4}e^{4v_2}\right)\\ &= 4\norm{u_1}_{L^\infty(\Sph^4_+)} + \log \int_{\Sph^4}e^{4v_2}.
\end{split}
\end{equation*}

Next, we apply \cite[Lemma 2.1]{djadli-malchiodi-Q1} to $v_2$, obtaining:
\begin{equation}\label{mt-local3}
	\begin{split}
\log\int_{\Omega} e^{4u} \leq C + 4\norm{u_1}_{L^\infty(\Sph^4)} + \frac{3}{\pi^2} \int_{\Sph^4_+} v_2 + \frac{1}{8\pi^2}\scal{P^4v_2,v_2}_{L^2(\Sph^4)},
\end{split}
\end{equation}
where $P^4=\Delta^2-2\Delta$ denotes the Paneitz operator on the closed sphere $\Sph^4$. Integrating by parts twice and using the fact that the derivatives of $v_2$ are compactly supported on $\Sph^4_+$
, we get
\begin{equation}\label{mt-local2}
	\begin{split}
	\scal{P^4v_2,v_2}_{L^2(\Sph^4)} &= \int_{\Sph^4} \(\Delta^2v_2\,v_2-2\Delta v_2\,v_2\) = \int_{\Sph^4} \((\Delta v_2)^2+2\abs{\nabla v_2}^2\) \\
	&=\int_{\Sph^4_+} \((\Delta v_2)^2+2\abs{\nabla v_2}^2\) = \scal{P^{4,3}v_2,v_2}_{L^2(\Sph^4_+)}.
	\end{split}
\end{equation}

By Leibniz's rule and Schwartz's inequality, for every $\e>0$ there exists a constant $C$ depending on $\e, \delta$ and $\Omega$ such that

\begin{equation}\label{mt-local1}
\begin{split}
\scal{P^{4,3}v_2,v_2}_{L^2(\Sph^4_+)}  \leq \int_{\Sph^4_+} \chi^2 \((\Delta u_2)^2+2\abs{\nabla u_2}^2\) + \e \scal{P^{4,3}u_2,u_2}_{L^2(\Sph^4_+)} + C \intI {u_2}^2.
\end{split}
\end{equation}
Combining \eqref{mt-local3} with \eqref{mt-local2} and \eqref{mt-local1}, up to relabelling $\e$, we obtain
\begin{equation}\label{mt-local4}
\begin{split}
\log\int_{\Omega} e^{4u} &\leq \frac{1}{8\pi^2}\int_{\Sph^4_+} \chi^2 \((\Delta u_2)^2+2\abs{\nabla u_2}^2\) + \e  \scal{P^{4,3}u_2,u_2}_{L^2(\Sph^4_+)} \\ &+ C \intI {u_2}^2 +4\norm{u_1}_{L^\infty(\Sph^4)} + \frac{3}{\pi^2} \int_{\Sph^4_+} \chi u_2 + C.
\end{split}
\end{equation}
Now, choose $\lambda>0$ to be a large enough eigenvalue of $P^{4,3}$ in $H^2(\Sph^4_+)$. Let $E_\lambda$ be the direct sum of the eigenspaces of $P^{4,3}$ with eigenvalues less or equal than $\lambda$, and denote by $\Pi$ and $\Pi^\perp$ the projections onto $E_\lambda$ and $E_\lambda^\perp$, respectively. We set

\begin{equation*}
	u_1 = \Pi(u) \quad\text{and}\quad u_2 = \Pi^\perp(u).
\end{equation*}

Since $E_\lambda$ is finite-dimensional and $\overline{u}{_{\Sph^4_+}}=0,$ the $L^\infty-$norm and the $L^2-$norm are equivalent, with a proportional factor that only depends on $\e$, $\delta$ and $\Omega$ (see \cite{djadli-malchiodi-Q1}). Moreover, by the equivalence between the norms given in Lemma \ref{norm-equiv} and the orthogonality between $u_1$ and $u_2$, it holds
\begin{equation}\label{mt-local6}
\begin{split}
\norm{u_1}_{L^\infty(\Sph^4_+)}&\leq C_1 \scal{P^{4,3}u_1,u_1}^{\frac{1}{2}}_{L^2(\Sph^4_+)}, \\ C\norm{u_2}_{L^2(\Sph^4_+)}^{2} &\leq \frac{C}{\lambda} \scal{P^{4,3}u_2,u_2}_{L^2(\Sph^4_+)} \leq \e \scal{P^{4,3}u_2,u_2}_{L^2(\Sph^4_+)}, \\
\int_{\Sph^4_+} u_2 &\leq C_2 \norm{u_2}_{L^2(\Sph^4_+)}\leq C_2 \norm{u}_{L^2(\Sph^4_+)}\leq C_3 \scal{P^{4,3}u,u}^{\frac{1}{2}}_{L^2(\Sph^4_+)},
\end{split}
\end{equation}
where $C,C_1,C_2$ and $C_3$ are positive constants and $\lambda$ has been taken in such a way that ${C\over \lambda}<\e$. Notice that $\lambda$ depends on $\e$ and $\delta$, but not $\Omega$. Finally, inserting \eqref{mt-local6} into \eqref{mt-local4}, and using Cauchy's inequality, our choices of $u_1$ and $u_2$ leads to
\begin{equation*}
	\begin{split}
		\log\int_{\Omega} e^{4u} &\leq \frac{1}{8\pi^2}\int_{\Sph^4_+} \chi^2 \((\Delta u_2)^2+2\abs{\nabla u_2}^2\) + \e\scal{P^{4,3}u_2,u_2}_{L^2(\Sph^4_+)} \\ &+ C_1 \scal{P^{4,3}u_1,u_1}^{\frac{1}{2}}_{L^2(\Sph^4_+)} + C_3 \scal{P^{4,3}u,u}^{\frac{1}{2}}_{L^2(\Sph^4_+)}+ C
		\\ &\leq\frac{1}{8\pi^2}\int_{\Omega^{\delta/ 2}} \((\Delta u_2)^2+2\abs{\nabla u_2}^2\)+\e\scal{P^{4,3}u,u}_{L^2(\Sph^4_+)} + C,
	\end{split}
\end{equation*}
which is what we wanted to prove.
\end{proof}
Similarly, one can prove the following local inequality.
\begin{proposition}\label{local-int2} Let $\Sigma\subset \Sph^4_+$ be such that $\Sigma\cap\de \Sph^4_+\neq \emptyset$ and $\Sigma\subset (\de \Sph^4_+)^\delta$ for some $\delta>0$. Then, for every $\e>0$, there exist a constant $0<C=C(\Omega,\e,\delta)$ and a large enough eigenvalue $\lambda=\lambda(\e,\delta)$ of $P^{4,3}$, such that
	
	\begin{equation}\label{local-ineq-2}
		\begin{split}
			\log \int_{\Sigma} e^{4u} \leq C + \frac{1}{4\pi^2}\int_{\Sigma^\delta}\((\Delta u^\dagger)^2+2\abs{\nabla u^\dagger}^2\)+\e\scal{P^{4,3}u,u}_{L^2(\Sph^4_+)}
		\end{split}
	\end{equation}
	for every $u\in H^2(\Sph^4_+)$ with $\frac{\de u}{\de\eta_g}=0,$ where $u^\dagger$ denotes the projection of $u$ to $E_\lambda^\perp$, being $E_\lambda$ the direct sum of the eigenspaces of $P^{4,3}$ with eigenvalues smaller or equal than $\lambda.$
\end{proposition}
Observe that the constant in \eqref{local-ineq-2} is doubled with respect to \eqref{local-ineq-1}. This is due to the fact that, in the proof of Proposition \ref{local-int1}, we can use the extension by zero of suitable truncations, staying under the hypotheses of \cite[Lemma 2.1]{djadli-malchiodi-Q1} for $\Sph^4$, and then go back to $\Sph^4_+$ keeping the best constant. In Proposition \ref{local-int2}, however, the set $\Sigma$ touches the boundary, and we need to make use of a symmetrization argument (see the proof of Proposition \ref{Prop:MT1}), or the inequality \eqref{MT1}, obtaining a factor two.

\medskip 

The following result is the localized version for Proposition \ref{prop:MT2-bis}, which follows from an easy adaptation of the arguments from Proposition \ref{local-int1} (see \cite{ndiaye2009}).

\begin{proposition}\label{local-int3} Let $S\subset \Sph^3$ and $\delta>0$. Then, for every $\e>0$, there exist a constant $0<C=C(\Omega,\e,\delta)$ and a large enough eigenvalue $\lambda=\lambda(\e,\delta)$ of $P^{4,3}$ in $\Sph^4_+$, such that
	
	\begin{equation*}
		\begin{split}
			\log \int_{S} e^{3u} \leq C + \frac{3}{16\pi^2}\int_{S^\delta}\((\Delta u^\dagger)^2+2\abs{\nabla u^\dagger}^2\)+\e\scal{P^{4,3}u,u}_{L^2(\Sph^4_+)}
		\end{split}
	\end{equation*}
	for every $u\in H^2(\Sph^4_+)$ with $\frac{\de u}{\de\eta_g}=0,$ where $u^\dagger$ denotes the projection of $u$ to $E_\lambda^\perp$, being $E_\lambda$ the direct sum of the eigenspaces of $P^{4,3}$ with eigenvalues smaller or equal than $\lambda.$
\end{proposition}

The immediate consequence of the inequalities \eqref{MT1} and \eqref{MT2-bis} is a lower bound for the energy functional $I$ defined by \eqref{functional}.

\begin{proposition}\label{lowerbound} There exists a constant $C\in\R$, independent of $\beta$ and $u$, such that $I(u,\beta)\geq C$ for every $\beta \in [0,4\pi^2]$ and every $u$ in the space of definition of $I_\beta$.
\end{proposition}
\begin{proof} Since we have $0<\intI Qe^{4u}$ and $0<\intB Te^{3u}$, we can bound $Q$ and $T$ by their $L^\infty-$norms, obtaining
	\begin{equation}\label{removeqt}
	\begin{split}
	\beta \log \intI Qe^{4u} &\leq \beta\log{\norm{Q}_{L^\infty(\Sph^4_+)}} + \beta \log \intI e^{4u}, \\
	\frac43(4\pi^2-\beta)\log \intB Te^{3u} &\leq \frac43(4\pi^2-\beta)\log{\norm{T}_{L^\infty(\Sph^3)}}+\frac43(4\pi^2-\beta)\log\intB e^{3u}.
	\end{split}
	\end{equation}
	Recalling the definition of $I(u,\beta)$ given by \eqref{functional}, it is enough to use \eqref{removeqt} and then interpolate \eqref{MT1} and \eqref{MT2-bis} to get 
	\begin{align*}
		I(u,\beta) &\geq \left(1-\frac{\beta}{4\pi^2}-\frac{4\pi^2-\beta}{4\pi^2}\right)\intI \left((\Delta_gu)^2+2\abs{\nabla_gu}^2\right) \\ &+\left(12-\frac{3\beta}{\pi^2}-\frac{3}{\pi^2}(4\pi^2-\beta)\right)\intI u \\ &-\beta\log{\norm{Q}_{L^\infty(\Sph^4_+)}} - \frac{4}3(4\pi^2-\beta)\log{\norm{T}_{L^\infty(\Sph^3)}} +\mathfrak f(\beta)  + C \\ &= \tilde C.
	\end{align*}
To remove the dependence of $\beta$ in the last line, we take the value $\beta\in [0,4\pi^2]$ (depending on $Q$ and $T$ but not $u$) that maximizes the function 
\begin{equation*}
\beta\mapsto \beta\log{\norm{Q}_{L^\infty(\Sph^4_+)}} + \frac{4}3(4\pi^2-\beta)\log{\norm{T}_{L^\infty(\Sph^3)}} +\mathfrak f(\beta).
\end{equation*}
\end{proof}

\section{Existence results
}\label{sec:lc}
This section is dedicated to the limiting cases $\beta=0$ and $\beta=4\pi^2$ which correspond to $Q=0$ and $T=0$ respectively. In particular, we will prove the existence of solution for the curvature prescription problem \eqref{problem} under these conditions.
Let us begin with the case $\beta=4\pi^2$, since the reasoning is slightly more involved than the one corresponding to $\beta = 0$. It is easy to check that the critical points of $I_{4\pi^2}$ in $\mathcal H^{4\pi^2}$ are weak solutions of

\begin{equation}\label{problemQ0-2}
\left\lbrace\begin{array}{rll}
\Delta^2u-2\Delta u +6 & = 8\pi^2 \dfrac{Qe^{4u}}{\intI Qe^{4u}} & \text{in}\hsp \Sph^4_+, \\[0.5ex]
-\dfrac{\de}{\de \nu}\Delta u &= 0 & \text{on} \hsp \Sph^3, \\[0.5ex]
\dfrac{\de u}{\de \nu} &= 0 & \text{on} \hsp \Sph^3.
\end{array}\right.
\end{equation}
Indeed, by the Gauss-Bonnet-Chern Theorem, i.e. \eqref{GBC-sphere}, we have $\intI Q e^{4u} = 4\pi^2$, which shows that \eqref{problemQ0-2} is equivalent to 
\begin{equation*}
	\left\lbrace\begin{array}{rll}
		\Delta^2u-2\Delta u +6 &= 2 Qe^{4u} & \text{in}\hsp \Sph^4_+, \\[0.5ex]
		-\dfrac{\de}{\de \nu}\Delta u &= 0 & \text{on} \hsp \Sph^3, \\[0.5ex]
		\dfrac{\de u}{\de \nu} &= 0 & \text{on} \hsp \Sph^3,
	\end{array}\right.
\end{equation*}
which is exactly \eqref{problem} with $T=0$. Our main result for this particular case is the following.
\begin{theorem}\label{th:existenceQ} Let $Q\in C^2(\Sph^4_+)$ be a non-negative, positive somewhere and $\mathcal G$-symmetric function. Then, if $Q=0$ on $\fixed$, the functional $I_{4\pi^2}$ defined by \eqref{functional4pi2} admits a global minimizer, which is a solution of $\eqref{problem}$ with $T=0$.
\end{theorem}

\begin{proof}
By Proposition \ref{lowerbound}, the functional $I_{4\pi^2}$ is bounded from below, and we can consider $u_n$ a minimizing sequence in $\HH^{4\pi^2}_{\mathcal G}=H^2_{\mathcal G}(\Sph^4_+)$. Thanks to Proposition \ref{pr:addition-constants}, we can assume without loss of generality that $\overline{(u_n)}_{\Sph^4_+}=0$.
\smallskip

Now, we distinguish the possible cases for $u_n$ given by Proposition \ref{pr:cases}. If $(i)$ holds (so $\mathcal F_\de \neq \emptyset$), the mass of the sequence $e^{4u_n}$, weighed against $Q$, concentrates around the fixed boundary points of $\mathcal G$, i.e., $\mathcal F_\de$. In this case, we proceed as follows: since $Q$ is a continuous function and $Q=0$ along $\mathcal F_\de$, given $\e>0$, it is possible to take $\delta>0$ and $t>0$ small enough in Lemmas \ref{lem:cov} and \ref{lem:cov2} so that $Q(x)<\e$ in $(U_\delta)^t$. Then, by $(i)$,
\begin{equation}\label{eq16-1} 
	\begin{split}
		\log \intI Qe^{4u_n} &\leq \log \int_{(U_\delta)^t} Qe^{4u_n}+\log {1\over \gamma} \leq \log \intI e^{4u_n}+\log \e + C.
	\end{split}
\end{equation}

\smallskip

Let $I^{Q,T}$ denote the energy functional $I$ given by \eqref{functional}, where we highlight the dependence on the prescribed curvatures $Q$ and $T$. By \eqref{eq16-1},

\begin{equation*}
	I_{4\pi^2}(u_n)\geq I^{1,0}(u_n,4\pi^2) -C\log \varepsilon+\tilde C \geq -C\log \varepsilon+\tilde C,
\end{equation*} 
where we have used that $I^{1,0}(u_n,\beta)$ is bounded from below by Proposition \ref{lowerbound}. If we take $\e$ sufficiently small, we negate that $u_n$ is a minimizing sequence, reaching a contradiction.

\medskip

Assume now that $(ii)$ holds. Then, we have
\begin{equation*}
	\intI Qe^{4u_n} \leq \frac{1}{\gamma} \int_{(A_i)^t} Qe^{4u_n}.
\end{equation*}
The previous inequality gives us the following lower bound for $I_{4\pi^2}(u_n)$:
\begin{equation}\label{eq16-2}
\begin{split}
I_{4\pi^2}(u_n)&\geq \scal{P^{4,3}u_n,u_n}_{L^2(\Sph^4_+)} + 12 \intI u_n - 4\pi^2\log\int_{(A_i)^t} Qe^{4u_n} + C.
\end{split}
\end{equation}
Now, we can use Proposition \ref{local-int2} twice, with $\Sigma=(A_i)^t$ and $\Sigma=\varphi((A_i)^t)$, obtaining
\begin{equation}\label{bound-part1}
	\begin{split}
4\pi^2\log\int_{(A_i)^t} Qe^{4u_n}&\leq \int_{(A_i)^t}\((\Delta u_n^\dagger)^2+2\abs{\nabla u_n^\dagger}^2\)+\e\scal{P^{4,3}u_n,u_n}_{L^2(\Sph^4_+)}+C,
\end{split}
\end{equation}
and, by the symmetry invariance, 
\begin{equation}\label{bound-part2}
\begin{split}
4\pi^2\log\int_{\varphi_i((A_i)^t)} Qe^{4u_n}&=4\pi^2\log\int_{(A_i)^t} Qe^{4u_n}\\&\leq \int_{\varphi_i((A_i)^t)}\((\Delta u_n^\dagger)^2+2\abs{\nabla u_n^\dagger}^2\)
+\e\scal{P^{4,3}u_n,u_n}_{L^2(\Sph^4_+)}+C,
\end{split}
\end{equation}
where the projection $u_n^\dagger$ is taken with respect to the maximum of the corresponding $\lambda$ for each choice of $\Sigma$. 

\smallskip

Adding \eqref{bound-part1} and \eqref{bound-part2} and using the fact that $(A_i)^t\cap \varphi((A_i)^t)=\emptyset$, one has
\begin{equation}\label{bound-part3}
	\begin{split}
8\pi^2 \log\int_{(A_i)^t} Qe^{4u_n} &\leq \int_{(A_i)^t\cup \varphi_i((A_i)^t)}\hspace*{-0.2cm} \((\Delta u_n^\dagger)^2+2\abs{\nabla u_n^\dagger}^2\)+\e\scal{P^{4,3}u_n,u_n}_{L^2(\Sph^4_+)}+C \\ &\leq \int_{\Sph^4_+}\((\Delta u_n^\dagger)^2+2\abs{\nabla u_n^\dagger}^2\)+\e\scal{P^{4,3}u_n,u_n}_{L^2(\Sph^4_+)}+C \\&\leq (1+\e)\scal{P^{4,3}u_n,u_n}_{L^2(\Sph^4_+)}+C,
\end{split}
\end{equation}
where in the last line we have used the fact that, by orthogonality, 
\begin{equation*}
\scal{P^{4,3}u_n^\dagger,u_n^\dagger}_{L^2(\Sph^4_+)}\leq \scal{P^{4,3}u_n,u_n}_{L^2(\Sph^4_+)}.
\end{equation*}
Finally, from \eqref{eq16-2} and \eqref{bound-part3} and using Lemma \ref{norm-equiv}, we derive the following lower bound for $I_{4\pi^2}(u_n)$:
\begin{equation}\label{bound-coerc}
\begin{split}
I_{4\pi^2}(u_n)&\geq \(1-\frac12-\e\)\scal{P^{4,3}u_n,u_n}_{L^2(\Sph^4_+)}- C\norm{u_n}_{H^2(\Sph^4_+)} + C \\ &\geq \tilde C\(\frac12-\e\)\norm{u_n}^2_{H^2(\Sph^4_+)}-C\norm{u_n}_{H^2(\Sph^4_+)}+C,
\end{split}
\end{equation}
where $\tilde C$ is a positive constant. Therefore, $u_n$ is bounded and a minimizer can be found using standard methods.

\medskip

To deal with (iii), we reason as in the previous case, using Proposition \ref{local-int1} (which already has the better constant $8\pi^2$), to obtain \eqref{bound-part3} and then \eqref{bound-coerc}.
\end{proof}
When $\mathcal F_\de =\emptyset$, we can drop the assumption that $Q\geq 0$, since the information provided by $Q$ becomes less relevant and we can use a simpler version of Proposition \ref{pr:cases} without the integral weight $Q$.

\begin{theorem}\label{th:q-nofixedpts} Let $Q\in C^2(\Sph^4_+)$ be a positive somewhere and $\mathcal G-$symmetric function. Assume that $\fixed =\emptyset$. Then, $I_{4\pi^2}$ admits a global minimizer in $\mathcal H^{4\pi^2}_{\mathcal G}$, which is a solution of \eqref{problem} with $T=0$.
\end{theorem}
\begin{proof} By Proposition \ref{lowerbound}, the functional $I_{4\pi^2}$ is bounded from below on $\mathcal H^{4\pi^2}_{\mathcal G}$, and we can consider a minimizing sequence $u_n$. Since $I_{4\pi^2}$ is invariant under the addition of constants (Proposition \ref{pr:addition-constants}), we can assume without loss of generality that $\overline{(u_n)}_{\Sph^4_+}=0$.
	
	\smallskip
Since $\fixed =\emptyset$, we take now $U_\delta=\emptyset$ and case $(i)$ in Proposition \ref{pr:cases} cannot happen. Thus, if we take $Q\equiv1$, Proposition \ref{pr:cases}  gives the existence of a $\gamma\in(0,1)$ such that the following alternative holds:
\begin{enumerate}
	\item[(ii)] either there exists $(A_i)^t\in\mathcal B$ such that 
	\begin{equation*}\int_{(A_i)^t}e^{4u_n}=\int_{\varphi_i\((A_i\)^t)}e^{4u_n}\geq \gamma\intI e^{4u_n}, \hsp\text{or}
	\end{equation*}
	\item[(iii)] 
	\begin{equation*}
		\int_{\Omega}e^{4u_n}\geq \gamma\intI e^{4u_n}.
	\end{equation*}
\end{enumerate}
Let $B$ denote $(A_i)^t$ or $\Omega$ depending on the case. Reasoning as in Theorem \ref{th:existenceQ}, we obtain the analogue of \eqref{bound-part3}, that is,
\begin{equation}\label{bound-1-nofixedpts}
\log \int_B e^{4u_n} \leq \(\frac{1}{8\pi^2}+\e\) \scal{P^{4,3}u_n,u_n}_{L^2(\Sph^4_+)} +C.
\end{equation}
Using \eqref{bound-1-nofixedpts}, Lemma \eqref{norm-equiv} and the estimate $Q\leq \norm{Q}_{L^\infty(\Sph^4_+)}$, we get
\begin{equation*}
\begin{split}
I_{4\pi^2}(u_n) &\geq \scal{P^{4,3}u_n,u_n}_{L^2(\Sph^4_+)} +12\intI u_n - 4\pi^2\log \int_B e^{4u_n}-\log{\norm{Q}_{L^\infty(\Sph^4_+)}}\\ \quad& \quad+\mathcal F(4\pi^2)+\log\frac1\gamma
\\ &\geq \(\frac12-\e\) \scal{P^{4,3}u_n,u_n}_{L^2(\Sph^4_+)}+12\intI u_n + C
\\ &\geq C\(\frac12-\e\) \norm{u_n}^2_{H^2(\Sph^4_+)}-C\norm{u_n}_{H^2(\Sph^4_+)} + C.
\end{split}
\end{equation*}
Hence, $I_{4\pi^2}$ is coercive on $\mathcal H^{4\pi^2}$ and it follows from standard minimization arguments that $u_n\to u_0$ strongly in $H^2_G(\Sph^4_+)$, with $I_{4\pi^2}(u_0)\leq I_{4\pi^2}(u)$ for every $u\in \mathcal H^{4\pi^2}$. 

\medskip

In this case, it is also necessary to check that $u_0\in \text{int}(\mathcal H^{4\pi^2})$. Assume by contradiction that $\intI Qe^{4u_0}=0$. Then, by compactness, we have  
\begin{equation*}
	\intI Qe^{4u_n}\to \intI Qe^{4u_0}=0.
\end{equation*}
(See \cite{Trudinger}.) Therefore, $I_{4\pi^2}(u_n)\to +\infty$, which contradicts that $u_n$ is a minimizing sequence.
\end{proof}
Similarly, when $\beta=0$, we can see that critical points of $I_0$ in $\mathcal H^0$ produce weak solutions of $\eqref{problem}$ with $Q=0$, this is,
\begin{equation*}
	\left\lbrace\begin{array}{rll}
		\Delta_g^2 u - 2\Delta_gu + 6 &= 0 & \text{in}\quad \Sph^4_+, \\[0.5ex]
		-\frac{\de}{\de \nu_g}\Delta_gu &= 2T e^{3u}  & \text{on}\quad \Sph^3, \\[0.5ex]
		\frac{\de u}{\de \nu_g} &=0 & \text{on}\quad \Sph^3.
	\end{array} \right.
\end{equation*}

Proposition \ref{local-int3} allows us to adapt the previous argument to this situation, and obtain the following existence result.

\begin{theorem}\label{th:existenceT} 
Let  $T\in C^2(\Sph^3)$ be a non-negative, positive somewhere and $\mathcal G-$symmetric function. Then, if $T=0$ on $ \mathcal F_\de $, the functional $I_0$ defined by \eqref{functional0} admits a global minimizer, which is a solution of \eqref{problem} with $Q=0$. 	
\end{theorem}

Again, when $\mathcal F_\de =\emptyset$, we can improve our result dropping the non-negativity asumption on $T$.
\begin{theorem}\label{th:t-nofixedpts} Let $T\in C^2(\Sph^3)$ be a positive somewhere and $\mathcal G-$symmetric function. Assume that $\mathcal F_\de  =\emptyset$. Then, $I_0$ admits a global minimizer in $\mathcal H^0_{\mathcal G}$, which is a solution of \eqref{problem} with $Q=0$.
\end{theorem}

\bigskip

\section{Proof of Theorem \ref{maintheorem}}
This section is devoted to the proof of existence of solution for the prescribed curvature problem \eqref{problem} in the most general case under the assumptions of Theorem \ref{maintheorem}. Since Theorems \ref{th:existenceQ} and \ref{th:existenceT} from Section \ref{sec:lc} already cover the case in which one of the curvatures is identically equal to zero, we will assume here that $Q$ and $T$ are both positive somewhere.

\medskip

 By Proposition \ref{lowerbound}, the energy functional $I$ defined by \eqref{functional} is bounded from below on $H^2_{\mathcal G}(\Sph^4_+)\times [0,4\pi^2]$, and we set
\begin{equation*}
\alpha=\inf\left\lbrace I(u,\beta): \beta\in (0,4\pi^2),\: u\in H_{\mathcal G}^2(\Sph^4_+)\right\rbrace.
\end{equation*}
Let $(u_n,\beta_n)$ be a minimizing sequence. By Proposition \ref{pr:addition-constants}, we can assume without loss of generality that $\overline{(u_n)}_{\Sph^4_+}=0.$ 

\medskip

First, it is convenient to point out the following identity:
\begin{equation}\label{order-minimum} 
	\alpha = \inf_{\beta\in(0,4\pi^2)} \left\lbrace \inf_{u\in H^2_{\mathcal G}(\Sph^4_+)} I(u,\beta) \right\rbrace.
\end{equation}
To prove it, let us call $\gamma$ the right-hand side of \eqref{order-minimum}. By definition of infimum, for any $\e>0$ there exists $\beta^*\in(0,4\pi^2)$ such that $\inf_{u\in H^2_{\mathcal G}(\Sph^4_+)} I(u,\beta^*)<\gamma+\e$. By taking infimum in $\beta^*\in(0,4\pi^2)$ and letting $\varepsilon\to 0,$ we obtain $\alpha\leq\gamma$. Conversely, for any $\e>0$ there exists $(u^*,\beta^*)\in H^2_{\mathcal G}(\Sph^4_+)\times(0,4\pi^2)$ such that $I(u^*,\beta^*)<\alpha+\varepsilon$. If we take infimum in $u^*\in H^2_\mathcal{G}(\Sph^4_+)$ and then in $\beta^*\in(0,4\pi^2)$, we obtain $\gamma < \alpha + \e$, and it suffices to let $\e\to 0$ to get $\alpha = \gamma$. 
\medskip

In view of \eqref{order-minimum}, we will first consider a fixed $\beta\in [0,4\pi^2]$ and work with the sequence $I(u_n,\beta)$, and then we will minimize with respect to $\beta.$ Our intention is to extend the results achieved for the cases $\beta=0$ and $\beta=4\pi^2$ in Theorems \ref{th:existenceQ} and \ref{th:existenceT}, that is, that the functional $I_\beta$ admits a global minimizer in $H^2(\Sph^4_+)$.	

\smallskip

Let $\beta\in(0,4\pi^2)$, consider $u_n$ a minimizing sequence in $H^2_{\mathcal G}(\Sph^4_+)\times(0,4\pi^2)$ and recall the energy functional defined by \eqref{functional}, i.e.,
\begin{equation*}
\begin{split}
I_\beta(u_n) &= \scal{P^{4,3}u_n,u_n}_{L^2(\Sph^4_+)}+12\intI u_n - \beta \log \intI Qe^{4u_n} \\ &-\frac43(4\pi^2-\beta)\log\intB Te^{3u_n}+\mathfrak f(\beta),
\end{split}
\end{equation*}
with $\mathfrak f$ given by \eqref{def:f}. As in the proof of Theorem \ref{th:existenceQ}, we need to consider all possible cases given by Proposition \ref{pr:cases} for the minimizing sequence. First, we observe that cases $(i)$ and $(a)$ never happen. Let us assume, for instance, that $(a)$ holds. Then, by continuity of $T$, since $T=0$ on $\fixed$, for any given $\e>0$ we find a $\delta>0$ small enough such that $T<\e$ on $U_\delta$. Then, by hypothesis, there exists $\lambda\in(0,1)$ such that

\begin{equation*}
\log \intB Te^{3u_n} < \log \int_{U_\delta} T e^{3u_n} + \log \frac1\lambda \leq \log \intB e^{3u_n} + \log\e + C.
\end{equation*}
The previous equation gives us the following lower bound for $I_\beta$:
\begin{equation*}
I_\beta(u_n)\geq I^{Q,1}(u_n,\beta) - \frac43(4\pi^2-\beta)\log\e + C \geq - \frac43(4\pi^2-\beta)\log \e + C,
\end{equation*}
where we have used Proposition \ref{lowerbound} to bound $I^{Q,1}(u_n,\beta)$ from below by a constant. This contradicts that $u_n$ is a minimizing sequence. Case $(i)$ can be discarded in a similar fashion. Therefore, $(b)$ and either $(ii)$ or $(iii)$ hold. 

\smallskip
Reasoning as in \eqref{bound-part3}, by $(b)$ and Proposition \ref{local-int3}, one has
\begin{equation}\label{info-b}
	\log \int_{A_i}Te^{3u_n} \leq \(\frac{3}{32\pi^2}+\e\) \scal{P^{4,3}u_n,u_n}_{L^2(\Sph^4_+)} + C.
\end{equation}
Moreover, for the integral term on $\Sph^4_+$ we have the exact analogue of \eqref{bound-part3};
\begin{equation}\label{info-ii}
\log \int_{B}Qe^{4u_n} \leq \(\frac{1}{8\pi^2}+\e\) \scal{P^{4,3}u_n,u_n}_{L^2(\Sph^4_+)} + C,
\end{equation}
where $B$ equals $(A_j)^t$ or $\Omega$ depending on whether case (i) or (ii) is satisfied. Continuing with this notation, interpolating \eqref{info-b} and \eqref{info-ii}, we derive the following lower bound for $I_\beta(u_n)$:

\begin{equation*}
\begin{split}
I_\beta(u_n) &\geq \scal{P^{4,3}u_n,u_n}_{L^2(\Sph^4_+)} + 12\intI u_n - \beta \log \int_B Qe^{u_n} \\ & -\frac43(4\pi^2-\beta)\log\int_{A_i} Te^{3u_n} + C \\ &\geq \(1-\frac{\beta}{8\pi^2}-\frac{4\pi^2-\beta}{8\pi^2}-\e\)\scal{P^{4,3}u_n,u_n}_{L^2(\Sph^4_+)}+ 12\intI u_n+C \\ &\geq \(\frac{1}{2}-\e\)\norm{u_n}^2_{H^2(\Sph^4_+)}-C\norm{u_n}_{H^2(\Sph^4_+)}+C.
\end{split}
\end{equation*}
Since the leading constant is independent of $\beta$, we conclude that $u_n$ is bounded in $H^2(\Sph^4_+)$. Following standard minimization arguments, one can show that $u_n\to u_\beta$ in $H^2_{\mathcal G}(\Sph^4_+)$, which is a global minimizer of $I_\beta$.

\medskip

Finally, we minimize in $\beta\in (0,4\pi^2)$. By compactness, it is clear that, up to considering a subsequence, one has $\beta_n\to \beta\in [0,4\pi^2]$. The proof concludes if we show that $\beta \in (0,4\pi^2)$.

\medskip 

Let us define $m:[0,4\pi^2]\to \R$ by $m(\beta)=\min\left\lbrace I_\beta(u):\:u\in H^2_\mathcal{G}(\Sph^4_+)\right\rbrace = I_\beta(u_\beta)$. By \eqref{order-minimum}, we have $m(\beta_n)\to m(\beta) = \alpha.$ 

\smallskip

Suppose by contradiction that $\beta_n\to 0$. Since $T,Q\geq 0$, we have $u_n\to u_0\in \mathcal H_\mathcal{G}$, and the following inequality holds
\begin{equation*}
m(0)-m(\beta_n)=I(u_0,0)-I(u_n,\beta_n)\geq I(u_0,0)-I(u_0,\beta_n).
\end{equation*}
Therefore,
\begin{equation*}
\begin{split}
	m(0)-m(\beta_n)&\geq I(u_0,0)-I(u_0,\beta_n)\\ &= \scal{P^{4,3}u_0,u_0}_{L^2(\Sph^4_+)}+12\intI u_0 -\frac{16\pi^2}{3}\log \intB Te^{3u_0} +\mathfrak f(0)\\&-\left(\scal{P^{4,3}u_0,u_0}_{L^2(\Sph^4_+)}+12\intI u_0-\beta_n \log\intI Qe^{4u_0}\right. \\&\left.-\frac43 (4\pi^2-\beta_n)\log \intB Te^{3u_0} +\mathfrak f(\beta_n)\right) 
	\\&=\frac{16\pi^2}{3}\log(4\pi^2)+\beta_n \log\left(\frac{\intI Qe^{4u_0}}{\intB Te^{3u_0}}\right) -\beta_n\log\beta_n \\&-\frac43(4\pi^2-\beta_n)\log(4\pi^2-\beta_n)-\frac{\beta_n}{3}.
\end{split}
\end{equation*} 
At first order, the last identity shows that the main term as $\beta_n\to 0$ is $-\beta_n\log\beta_n$. Hence, $m(0)-m(\beta_n)>0$, which contradicts that $\beta=0$ is a minimizer.

\smallskip

Analogously, if we suppose $\beta_n\to 4\pi^2$ and we write $\beta_n = 4\pi^2-\gamma_n$, with $\gamma_n\to 0$, we would have
\begin{equation*}
\begin{split}
m(4\pi^2)-m(\beta_n)&\geq I(u_{4\pi^2},4\pi^2)-I(u_{4\pi^2},\beta_n)
\\&=
4\pi^2\log(4\pi^2)+\gamma_n \log\left(\frac{\intI Qe^{4v_0}}{(\intB Te^{3v_0})^{4/3}}\right) -\beta_n\log\beta_n \\&-\frac43\gamma_n\log\gamma_n+\frac{\gamma_n}{3},
				\end{split}
\end{equation*} 
and the same conclusion holds.

\end{document}